\documentclass[11pt,a4paper]{amsart}

\usepackage{graphicx}
\usepackage{booktabs}
\usepackage{amssymb}
\usepackage{hyperref}
\hypersetup{
bookmarks=true,
colorlinks=true,
citecolor=[rgb]{0,0,0.5},
linkcolor=[rgb]{0,0,0.5},
urlcolor=[rgb]{0,0,0.5},
pdfpagemode=UseNone,
pdfstartview=FitH,
pdfdisplaydoctitle=true,
pdftitle={Band-limited maximizers, II},
pdfauthor={Barker, Thiele, Zorin-Kranich},
pdflang=en-US
}
\usepackage{mathtools}
\usepackage{microtype}
\usepackage{tikz}
\usepackage{pgfplots}

\usetikzlibrary{pgfplots.groupplots}
\usetikzlibrary{external}
\pgfplotsset{compat=1.13}

\usepackage[
  backend=biber,
  style=alphabetic,
  ]{biblatex}
  \DeclareFieldFormat[article, inbook]{title}{\emph{#1}}
\allowdisplaybreaks[4]

\def\R{\mathbb{R}}

\def\Z{\mathbb{Z}}

\def\bfn{\mathbf{n}}
\def\bfm{\mathbf{m}}

\DeclarePairedDelimiter\scp{\lparen}{\rparen}    
\DeclarePairedDelimiter\scb{\lbrace}{\rbrace}    
\DeclarePairedDelimiter\scs{\lbrack}{\rbrack}    

\DeclarePairedDelimiter\abs{\lvert}{\rvert}
\DeclarePairedDelimiter\norm{\lVert}{\rVert}

\DeclarePairedDelimiterX\innerp[2]{\langle}{\rangle}{#1,#2}
\providecommand\given{}
\newcommand\SetSymbol[1][]{%
\nonscript\:#1\vert{}
\allowbreak{}
\nonscript\:
\mathopen{}}
\DeclarePairedDelimiterX\Set[1]\{\}{%
\renewcommand\given{\SetSymbol[\delimsize]}
#1
}

\newcommand*{\id}[1]{\,\mathrm{d}#1}

\newtheorem{theorem}{Theorem}

\newtheorem{lemma}[theorem]{Lemma}

\title[Band-limited maximizers, II]{Band-limited maximizers for a Fourier extension inequality on the circle, II}

\author[Barker]{James Barker}

\address[Barker]{Fraunhofer SCAI, Schloss Birlinghoven, 53754 Sankt
  Augustin, Germany}

\email{james.barker@scai.fraunhofer.de}

\address[Barker]{Institute for Numerical Simulation, Endenicher
  Allee 19B, 53115 Bonn, Germany}

\author[Thiele]{Christoph Thiele}
\author[Zorin-Kranich]{Pavel Zorin-Kranich}

\address[Thiele, Zorin-Kranich]{Hausdorff Center for Mathematics,
  53115 Bonn, Germany}

\email{thiele@math.uni-bonn.de}
\email{pzorin@math.uni-bonn.de}

\begin{document}
\begin{abstract}
  Among the class of functions on the circle with Fourier modes up to
  degree $120$, constant functions are the unique real-valued
  maximizers for the endpoint Tomas-Stein inequality.
\end{abstract}

\maketitle

\section{Introduction}

This article continues the investigation in~\cite{oliveira2019band} of
extremizers of the classical Tomas-Stein~\cite{tomas1975restriction}
Fourier extension inequality on the circle in classes of band limited
functions.  We improve the main theorem in~\cite{oliveira2019band},
which concerns functions with Fourier modes up to degree $30$, towards
degree $N=120$.

\begin{theorem}\label{thm:main}
  Let $f\in L^2\scp*{\mathbb{S}^1}$ be real-valued.  Assume that
  $\widehat{f}_n=0$ for all $n>120$. Then
  \[
    \Phi\scp*{f}\leq\Phi\scp*{\bf{1}},
  \]
  with equality if and only if $f$ is constant.
\end{theorem}

Here, $\mathbb{S}^1$ is the unit circle in the complex plane, and
$\widehat{f}_n$ are the coefficients in the Fourier series
\[f\scp*{\omega}=\sum_{n\in {\mathbb{Z}}} \widehat{f}_n\omega^n .\] The
Tomas-Stein functional
\begin{equation*}
  \Phi\scp*{f} \vcentcolon=
  \norm*{\widehat{f\sigma}}^6_{L^6\scp*{\R^2}}
  \norm*{f}_{L^2\scp*{\mathbb{S}^1}}^{-6}
\end{equation*}
is the sixth power of the norm quotient for the Fourier extension map
\begin{equation*}
  \widehat{f\sigma}\scp*{x} \vcentcolon=
  \int_{\mathbb{S}^1} f\scp*{\omega}\,e^{-i x\cdot\omega}\id{\sigma_\omega},
  \;\;\;
  \scp*{x\in\R^2},
\end{equation*}
where we identify the complex plane with the Euclidean plane $\R^2$
when we take the dot product $x\cdot \omega$, and $\sigma$ is the
arclength measure on the circle. The constant function ${\bf 1}$ is
conjectured to extremize the functional $\Phi$ among all functions in
$L^2\scp*{\mathbb{S}^1}$.

We refer to~\cite{carneiro2017sharp},~\cite{oliveira2019band},
and~\cite{foschi2017some} for further background on the sharp Fourier
restriction and extension problems.  In particular, it is known that
real-valued maximizers of the functional $\Phi$ do not change sign and
are antipodally symmetric. Hence, it suffices to show the variant of
Theorem~\ref{thm:main} for non-negative and antipodally symmetric
functions as in~\cite{oliveira2019band}.

Since $f$ in Theorem~\ref{thm:main} is assumed to be real-valued, we
have $\hat{f}_{n}=\overline{\hat{f}_{-n}}$, so that the Fourier modes
vanish for $\abs*{n}>120$.  Theorem~\ref{thm:main} thus concerns finite set
of Fourier modes and turns the extremizing problem into a finite
dimensional problem. This makes it accessible to numerical
computation. As in~\cite{oliveira2019band}, the theorem is reduced to
demonstrating positive semi-definiteness of each of the $3N/2+1$
matrices
\begin{equation}\label{qmatrices}
  {\scp*{Q_{{\bf m}, {\bf n}}}}_{{\bf m}, {\bf n}\in X_{D}}
\end{equation}
with $0\le D \le 3N$ an even number,
\[X_D=
\Set*{ \scp*{m_1, m_2, m_3}\in {\scp*{2\mathbb{Z}}}^3\setminus \{\scp*{0, 0, 0}\}
  \given \substack{m_1+m_2+m_3=D,\\
    \abs*{m_1}, \abs*{m_2}, \abs*{m_3}\le N,\\
    m_{1} \leq m_{2} \leq m_{3} }},
\]
and, writing $S_3$ for the group of permutations of three elements and
\[
  {\bf n}_\sigma=\scp*{n_{\sigma\scp*{1}}, n_{\sigma\scp*{2}}, n_{\sigma\scp*{3}}}, \quad
  \sigma\in S_{3},
\]
we set
\begin{align*}
Q_{{\bf m}, {\bf n}} &\vcentcolon= \frac 16 \sum_{\sigma\in S_3} \scp*{R_{{\bf m}, {\bf n_\sigma}}-L_{{\bf m}, {\bf n_\sigma }}},\\
L_{{\bf m}, {\bf n}} &\vcentcolon= 2I_{{\bf m}, -{\bf n}}+\sum_{\sigma\in S_3}
  I_{{\bf m}, {-\bf n}+{\scp*{1, -1, 0}}_\sigma},\\
R_{{\bf m}, {\bf n}} &\vcentcolon= 2I_{{\bf m}-{\bf n}, \scp*{0, 0, 0}}+\sum_{\sigma\in S_3} I_{{\bf m}-{\bf n}, {\scp*{1, -1, 0}}_\sigma},\\
  I_{{\bf m}, {\bf n}} &\vcentcolon= I_{\bf{k}} \vcentcolon= \int_0^\infty   \prod_{j=1}^6 J_{k_j}\scp*{r} r \id{r},
\end{align*}
where $\bf{k} \in \Z^6$ in the final definition is the concatenation
of $\bf{m}$ and $\bf{n}$, and the Bessel function $J_k$ is defined by
\[
  \int_{\mathbb{S}^1} \omega^k e^{-ix\cdot\omega} \id{\sigma_\omega}
  = 2\pi {\scp*{-i}}^k J_k\scp*{\abs{x}}{\scp*{x/\abs{x}}}^k.
\]
Note that in the formulas for $L_{{\bf m}, {\bf n}}$ and
$R_{{\bf m}, {\bf n}}$ given in~\cite{oliveira2019band}, ${\bf n}$
should be replaced by $-{\bf n}$ on the right-hand side. This error is
corrected here.

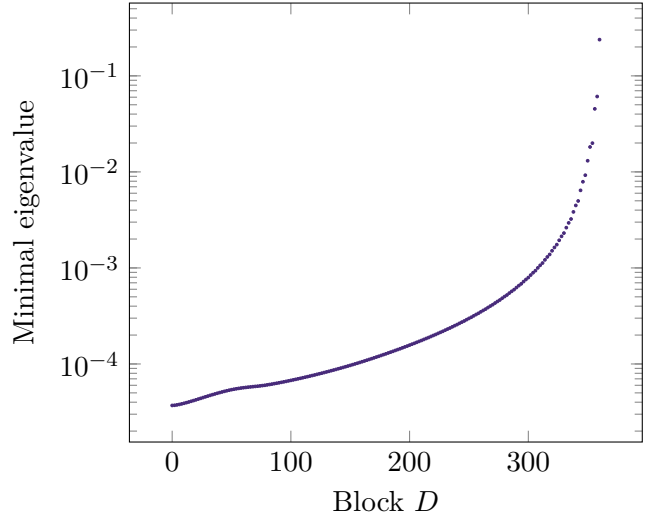
\begin{figure}
  {\centering
    \begin{tikzpicture}
      \begin{semilogyaxis}[
        scale only axis,
        width=0.8*\axisdefaultwidth,
        xlabel=Block $D$,
        ylabel=Minimal eigenvalue,
        cycle list={[indices of colormap={2} of colormap/viridis]},
        ]
        \addplot+ [only marks, mark size=0.5pt] table {figure_data/figure_1.dat};
      \end{semilogyaxis}
    \end{tikzpicture}
  }
  \caption{Minimal eigenvalue for each block $0 \leq D \leq 360$, $N=120$.}%
  \label{fig:smallest_evs}
\end{figure}

The main computational task in the proof of the theorem is the
numerical approximation of the various integrals $I_{\bf k}$. The
number of such integrals increases as the fifth power of $N$;
approximately $2.1 \times 10^5$ distinct integrals (up to sign and
permutation of ${\bf k}$) must be calculated for the $N=30$ case,
increasing to approximately $1.6 \times 10^8$ for the $N=120$ case. In
Section~\ref{numest}, we describe carefully the numerical scheme used
to approximate each ${I_{\bf k}}$ and estimate the associated error.
This scheme is an improved version of the scheme given
in~\cite{oliveira2019band}, requiring fewer arithmetic operations per
integral. We also discuss the adjustment of the parameters in the
scheme for arbitrary $N$.

Due to the high level of precision required, the various calculations
used in the proof were performed using arbitrary-precision arithmetic
as implemented in the Arb numerical library~\cite{johansson2017arb};
when used carefully, such arithmetic provides rigorous error bounds
for the output of calculations. The calculation code was written in
C++, employing hybrid parallelization through the use of both OpenMP
and MPI\@. Computations were performed in parallel using 94 nodes of a
modern compute cluster, requiring approximately 16 hours of wall time.
Weights and points used for Gauss-Legendre quadrature were calculated
using Mathematica~\cite{Mathematica} to high precision and imported
into the main calculation code by hand. The calculations can be
reproduced using the code given in the supporting information; a copy
of the output of this code, showing the approximated eigenvalues of
the matrices~\eqref{qmatrices}, is also given. Some results (plots,
etc.) outside the scope of the main proof were calculated by
post-processing the calculated integrals ${\tilde{I}_{\bf k}}$ using
standard numerical software, as the precision requirements are in
context less stringent.

\begin{figure}
  {\centering
    \begin{tikzpicture}
      \begin{loglogaxis}[
        scale only axis,
        width=0.8*\axisdefaultwidth,
        xlabel=$N$,
        ylabel=Minimal eigenvalue,
        scaled ticks=false,
        xtick={20, 40, 60, 80, 100, 120},
        xticklabel={
          \pgfkeys{/pgf/fpu=true}
          \pgfmathparse{exp(\tick)}%
          \pgfmathprintnumber[fixed relative]{\pgfmathresult}
          \pgfkeys{/pgf/fpu=false}
        },
        cycle list={[indices of colormap={2,9} of colormap/viridis]},
        ]
        \addplot+[only marks, mark size=0.5pt] table {figure_data/figure_2.dat};
        \addlegendentry{Eigenvalues}

        \addplot+[
        domain=20:120,
        dashed,
        thick,
        no markers,
        ] {0.153 * x^-1.74};
        \addlegendentry{$0.15N^{-1.74}$};
      \end{loglogaxis}
    \end{tikzpicture}
  }
  \caption{Minimal eigenvalues for the $D=0$ matrix in~\eqref{qmatrices}
    with $20 \leq N \leq 120$.}%
  \label{fig:smallest_ev_D_0_N_varies}
\end{figure}
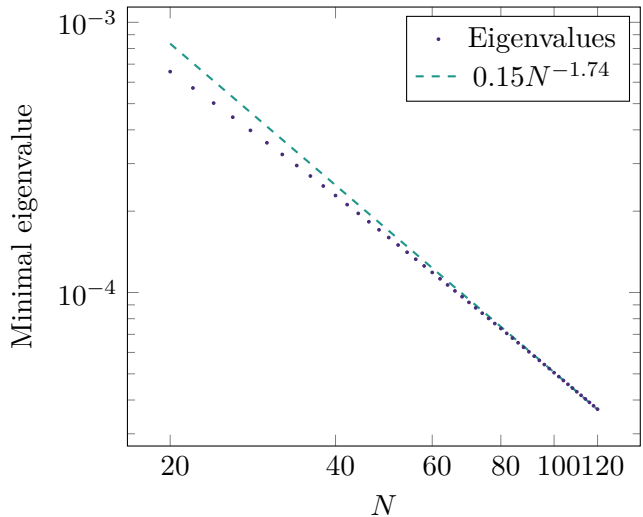

In Section~\ref{numresults}, we describe the results of these
computations, give an assessment of the eigenvalues, and discuss
various plots of matrix entries and eigenfunctions. The quality of the
experimentally-obtained information about the matrices
in~\eqref{qmatrices} and their eigenvalues and eigenfunctions has
substantially improved compared to the computations
in~\cite{oliveira2019band}. Phenomena are seen with much better
resolution and allow further investigation. In particular, we obtain
positivity of all eigenvalues in question, as shown to sufficient
accuracy in Figure~\ref{fig:smallest_evs}, which implies
Theorem~\ref{thm:main} by the reductions described
in~\cite{oliveira2019band}.\!\footnote{During the preparation of the
  final version of this paper, the authors became aware of a minor but
  non-negligible error in previous preprint versions, namely that an
  incorrect value of $T$ was used in the numerical calculation of the
  tail approximation $\tilde{I}^{T, \infty}_{\bf k}$ in
  \eqref{eq:SplitII}. The error has been corrected, and all values,
  plots, etc.\ have been recalculated and updated accordingly.}

Figure~\ref{fig:smallest_evs} shows that $D=0$ has the smallest
eigenvalue among the matrices in~\eqref{qmatrices} for $N=120$. It is
cheap to conjecture that the analogous statement holds for arbitrarily
large $N$. In addition, Figure~\ref{fig:smallest_ev_D_0_N_varies}
suggests the conjecture that the smallest eigenvalue of the matrix
$D=0$ has a power decay, possibly of the order $N^{-7/4}$, and is in
particular positive. These conjectures would imply the analogue to
Theorem~\ref{thm:main} for arbitrarily large $N$, and with it the
general conjecture about constant functions maximizing the Tomas-Stein
functional in $L^2\scp*{\mathbb{S}^1}$.

\section{Numerical estimation of \texorpdfstring{$I_{\bf k}$}{I{\_}k} and error bounds}%
\label{numest}

As in~\cite{oliveira2019band}, we approximate integrals $I_{\bf k}$ by
quantities $\tilde{I}_{\bf k}$ defined by approximating schemes.  We
split
\begin{equation}\label{eq:SplitI}
  \begin{split}
    I_{\bf k}
    &= {I}^{0, S}_{\bf k} + {I}^{S, T}_{\bf k} +{I}^{T, \infty}_{\bf k}\\
    &= \int_0^S \prod_{j=1}^6 J_{k_j}\scp*{r} r \id{r}
    +\int_S^T \prod_{j=1}^6 J_{k_j}\scp*{r} r \id{r}
    +\int_T^\infty \prod_{j=1}^6 J_{k_j}\scp*{r} r \id{r},
  \end{split}
\end{equation}
and correspondingly combine
\begin{equation}\label{eq:SplitII}
  \tilde{I}_{\bf k}=\tilde{I}^{0, S}_{\bf k}+\tilde{I}^{S, T}_{\bf k}+
  \tilde{I}^{T, \infty}_{\bf k},
\end{equation}
where the first two terms are quadrature rules with different
parameters approximating the corresponding compact integrals
in~\eqref{eq:SplitI}, and the third term is an exact integral over an
approximation to the integrand using asymptotic expansion. As opposed
to~\cite{oliveira2019band}, we use Gauss-Legendre quadrature instead
of Newton-Cotes quadrature on the first two pieces, and we use more
terms of the asymptotic expansion on the third integral. The cutoffs
$S$ and $T$ should be chosen well for numerical speed and accuracy.
Our discussion of the error bounds requires
$S\ge 0.95 N^{ 3/2}\ln\scp*{N}+1$ and $T>10N^2$. For $N=120$, in
accordance with these conditions, we choose $S=6{,}000$ and
$T=150{,}000$.

\subsection{Gauss-Legendre quadrature on the intervals
  \texorpdfstring{$\scs*{0, S}$}{[0,S]} and
  \texorpdfstring{$\scs*{S, T}$}{[S,T]}}

We cut each of the intervals $\scs*{0, S}$ and $\scs*{S, T}$ into
small intervals of constant length and use Gauss-Legendre quadrature
with $n=12$ points on each of these small intervals. We will estimate
the error of the quadrature using Lemma~\ref{abgauss} below. We first
quickly review the theory behind this lemma.

In $L^2\scp*{\scs*{-1, 1}}$, the even or odd real monic polynomial
\begin{equation}\label{rodrigues}
\frac{n!}{\scp*{2n}!} \partial_{x}^n {\scp*{x^2-1}}^n=\vcentcolon{} \prod_{i=1}^n \scp*{x-x_i}=\vcentcolon{}p_n\scp*{x}
\end{equation}
is orthogonal to all polynomials of lower degree (this is the Rodrigues formula for the Legendre polynomials with a different constant factor). This is seen by
$n$-fold partial integration, with boundary terms vanishing due to the
structure of $p_n$. The zeroes $x_1, \dots, x_n$ are distinct and
contained in $\scp*{-1, 1}$, or else there was a lower degree polynomial
with the same sign as $p_n$ on $\scs*{-1, 1}$, contradicting
orthogonality. The linear combination $x p_n -p_{n+1}$ has degree at
most $n$ and is orthogonal to all polynomials $p_k$ with $k<n-1$.  By
parity consideration, it is a multiple of $p_{n-1}$ with factor
determined by examination of the highest order coefficient:
\begin{align*}\label{ttrec}
  x p_n -p_{n+1}
  &= \scs*{\frac{-n^2\scp*{n-1}}{\scp*{2n}\scp*{2n-1}}-\frac{-{\scp*{n+1}}^2n}{\scp*{2n+2}\scp*{2n+1}}}p_{n-1}\\
  &= \frac {n^2}{\scp*{2n+1}\scp*{2n-1}} p_{n-1}.
\end{align*}
Pairing with $p_{n-1}$, using orthogonality relations, identifying the
$n$-th factor of the Wallis product, and solving the recursion yields
\[
  \int_{-1}^1 p_n^2\scp*{x}\id{x}
  = \frac 14 \frac {{\scp*{2n}}^2}{\scp*{2n-1}\scp*{2n+1}}
  \int_{-1}^1 p_{n-1}^2\scp*{x}\id{x}
  \le 2^{-2n}\pi.
\]

\begin{lemma}[Gauss-Legendre quadrature]\label{gauss}
  There are weights $w_i$ such that, for every function $f$ that is
  $2n$ times continuously differentiable on $\scs*{-1, 1}$, we have
  \begin{equation}\label{gaussbound}
    \abs*{\int_{-1}^1 f\scp*{y}\id{y}
      - \sum_{i=1}^n w_i f\scp*{x_i}}\le \pi 2^{-2n}
    \sup_{\xi \in \scs*{-1, 1}}\frac{\abs*{f^{\scp*{2n}}\scp*{\xi}}}{\scp*{2n}!}
  \end{equation}
\end{lemma}

\begin{proof}By regularity of the Vandermonde determinant, there are
  weights $w_i$ such that the left-hand-side of~\eqref{gaussbound}
  vanishes if $f$ is a polynomial of degree at most $n-1$. The
  left-hand-side also vanishes evidently for all polynomials of the
  form $x^k p_n$ with $k\le n-1$, and thus for all polynomials of
  degree at most $2n-1$.  For arbitrary $f$ as in the lemma, let $h$
  be the polynomial of degree $<2n$ such that $f-h$ vanishes of second
  order at all points $x_i$.  Then the function $k=\scp*{f-h}p_n^{-2}$ is
  continuous.  We estimate the left-hand-side of~\eqref{gaussbound}
  as:
  \begin{align*}
    \int_{-1}^1 \scp*{f-h}\scp*{y}\id{y}
    &= \int_{-1}^1 k\scp*{y}p_n{\scp*{y}}^2\id{y}\\
    &\le \abs*{k\scp*{x}} \int_{-1}^1 p_n{\scp*{y}}^2\id{y}\\
    &\le \pi 2^{-2n} \abs*{k\scp*{x}}
  \end{align*}
  for some $x\in \scs*{-1, 1}$. By Rolle's theorem, there is a
  $\xi\in \scs*{-1, 1}$ with
  \[
    0
    =\partial_\xi^{2n}(f-h-k\scp*{x}{p}_n^2)\scp*{\xi}
    =f^{\scp*{2n}}\scp*{\xi}- k\scp*{x} \scp*{2n}!
    \qedhere
  \]
\end{proof}
\begin{lemma}\label{abgauss}
  Assume $f$ is analytic on the union of balls of radius $1$ about
  each of the points of the interval $\scs*{a, b}$ and bounded by $F$ on
  this union. Then
  \[
    \abs*{
      \int_a^b f\scp*{x}\id{x}
      -\sum_{i=1}^n w_{i}\frac{b-a}2 f\scp*{\frac {a+b}2 +\frac {b-a}2 x_i}}
    \le  \pi 2^{-2n}{\scp*{\frac{\abs*{b-a}}2}}^{2n+1} F .\]
\end{lemma}
\begin{proof}
  This follows by translation and dilation of $f$ from the case
  $\scs*{a, b}=\scs*{-1, 1}$ with balls of radius $1$ replaced by
  balls of radius $2/\abs*{b-a}$.  The estimate in case $\scs*{-1, 1}$
  follows from Lemma~\ref{gauss} when estimating
  $\frac{f^{\scp*{2n}}\scp*{\xi}}{2n!}$ with Cauchy's integral formula as an
  average of the analytic function $f$ over the circle of radius
  $2/\abs*{b-a}$ about the point $\xi$.\end{proof}

On the interval $\scs*{0, S}$, we use the bound
\[
  \abs{J_n\scp*{z}}\le e^{\abs{\Im\scp*{z}}}
\]
on the strip $\Im\scp*{z}\le 1$, obtained as reviewed in Section 2
of~\cite{e2015estimates} from the integral representation
\[
  J_n\scp*{z}=\frac 1{2\pi}\int_0^{2\pi} e^{iz\sin\scp*{\theta}} e^{in\theta}\id{\theta}.
\]
Hence, the function
\begin{equation}\label{functionf}
  f\scp*{z}=z \prod_{i=1}^6 J_{n_i}\scp*{z}
\end{equation}
is bounded in absolute value by $\abs*{z}e^{6} $ on this strip.  Cutting
the interval $\scs*{0, S}$ into $K$ intervals of length $2d=S/K$ and
using Gauss-Legendre quadrature as in Lemma~\ref{abgauss} on each
interval gives
\begin{equation}\label{eq:gauss-nueric-error-on-[0, S]}
  \begin{split}
  \abs*{I_{\bf k}^{0, S}-\tilde{I}_{\bf k}^{0, S}}
  &= \abs*{
    \int_{0}^S{f}\scp*{z}\id{z}
  - \sum_{j=0}^{K-1} d \sum_{i=1}^n{w}_i f(d\scp*{2j+1}+dx_i)}\\
  &\le \pi 2^{-2n} d^{2n+1} e^6 \sum_{j=0}^{K-1} (2d\scp*{j+1}+1)\\
  &\le \pi 2^{-2n-1} d^{2n} e^6 \int_0^S \scp*{x+2d+1} \id{x}\\
  &\le \pi 2^{-2n-1} d^{2n}e^6 {\scp*{S+2d+1}}^2\\
  &\le 0.000038 \times {\scp*{S+2d+1}}^2 d^{24}\\
  &\le 0.1 \times 10^{-10}.
  \end{split}
\end{equation}
Here we have used $n=12$ in the penultimate and $S=6{,}000$ and
$d=0.25$ in the ultimate inequality.  Note that this estimate does not
depend on particular assumptions on $S$, and our parameters lead to an
algorithm with $144{,}000$ evaluations of the integrand.

On the interval $\scs*{S, T}$, we recall from~\cite[Section
2]{e2015estimates} the following representation for $J_n$, which
arises through the change of variables $t=\cos\scp*{\theta}$ from the
Poisson integral:
\[
  J_n\scp*{z}
  = \frac{{\scp*{z/2}}^n}{\Gamma\scp*{n+1/2}\Gamma\scp*{1/2}}
  \int_{-1}^1 \cos\scp*{z t}{\scp*{1-t^2}}^{n-1/2} \id{t}.
\]
We split $J_n=\frac 12 \scp*{J_n^++J_n^-}$, where
$J_n^+\scp*{\overline{z}}=\overline{J_n^-\scp*{z}}$ and
\[
  J_n^+\scp*{z}
  =\frac  {{\scp*{z/2}}^n}{\Gamma\scp*{n+1/2}\Gamma\scp*{1/2}}
  \int_{-1}^1 e^{i z t}{\scp*{1-t^2}}^{n-1/2} \id{t}.
\]
Indeed, $J_n=J_n^+$ due to symmetry of the weight. A change of the contour integral leads to
\begin{equation}\label{jplus}
  J_n^+\scp*{z}=\frac {{\scp*{2\pi z}}^{-1/2}}{ \Gamma\scp*{\nu+1}}
  \int_0^\infty e^{-u} u^{\nu}
  \scs*{e^{-i\omega}{\scp*{1-\frac {iu}{2z}}}^\nu
    +  e^{i\omega}{\scp*{1+\frac {iu}{2z}}}^\nu} \id{u}
\end{equation}
with the abbreviations
\begin{align*}
  \nu &\vcentcolon=n-\frac{1}{2},\\
  \omega &\vcentcolon=z-\frac \pi 4 -\frac{n\pi}2.
\end{align*}
We split further $J_n^+=\frac 12 \scp*{J_n^{++} + J_n^{+-}}$ with
$J_n^{+-}\scp*{\overline{z}}=\overline{J_n^{++}\scp*{z}}$ and
\begin{equation}\label{jplusplus}
  J_n^{++}\scp*{z}={\scp*{\frac 2{\pi z}}}^{1/2} \frac {e^{i\omega}}{\Gamma\scp*{\nu+1}}
  \int_0^\infty e^{-u} u^{\nu}{\scp*{1+\frac {iu}{2z}}}^\nu \id{u}.
\end{equation}

\begin{lemma}\label{stlemma}
  Assume $N\ge 20$ and $0\le n\le N$. Assume
  $\abs*{\Re\scp*{z}}> 0.95 N^{\frac 32}\ln\scp*{N}$ and $\abs*{\Im\scp*{z}}\le 1$.  Then
  \[\abs*{J_n^{++}\scp*{z}} \le 3.36 \abs*{z} ^{-1/2}.\]
  The analogous estimate holds for $J_n^{+-}$, $J_n^+$, $J_n^-$, and $J_n$.
\end{lemma}

\begin{proof}
  We first estimate the part of the integral in~\eqref{jplusplus} from
  $0$ to $2N\ln\scp*{N}$.  We estimate in this range
  \begin{equation*}
    \begin{split}
      \frac 12
      &\le \abs*{1 + \frac{iu}{2z}}=
      \abs*{1 + \frac{u\Im\scp*{z}}{2\abs*{z}^2} +  i \frac{ u\Re\scp*{z}}{2\abs*{z}^2}}\\
      &\le \sqrt{{\scp*{1+\frac 1 {0.9 N^2 \ln\scp*{N} }}}^2+\frac 1 {0.9 N}}\\
      &\le 1+\frac 1{N},
    \end{split}
  \end{equation*}
  where the lower bound by $\frac 1 2$ will only be used if
  $\nu$ is negative, that is $\nu =-1/2$.  We obtain
  \begin{equation*}
    \begin{split}
    \abs*{
      \int_0^{2N\ln\scp*{N}} e^{-u} u^{\nu} {\scp*{1+\frac {iu}{2z}}}^\nu \id{u}}
    &\le e \int_0^{\infty} e^{-u} u^{\nu} \id{u}\\
    &= e \Gamma\scp*{\nu+1}.
    \end{split}
  \end{equation*}
  Turning to the part of the integral in~\eqref{jplusplus} from
  $2N\ln\scp*{N} $ to $\infty$, we estimate
  \[
    \abs*{1\pm \frac{iu}{2z}}^\nu\le {\scp*{ \frac u N }}^{\nu+2}.
  \]
  Hence we have
  \[
    \abs*{
      \int_{2N\ln\scp*{N}}^\infty e^{-u} u^{\nu}{\scp*{1+\frac {iu}{2z}}}^\nu} \id{u}
    \le
    e^{-N\ln\scp*{N}} \int_{0}^\infty e^{-u/2} u^\nu {\scp*{\frac u N}}^{\nu+2} \id{u}
  \]
  \[
    = N^{-{N}-\nu-2} 2^{2\nu+3}  \Gamma\scp*{2\nu + 3} \le \frac 1{100} \Gamma\scp*{\nu+1}.
  \]
  With $\abs*{e^{i\omega}}\le \cosh\scp*{\abs*{\Im z}}\le \cosh\scp*{1}$ it follows that

  \[\abs*{J_n^+\scp*{z}}\le {\scp*{\frac 2{\pi \abs*{z}}}}^{1/2} \cosh\scp*{1}\scp*{\frac 1{100} +  e } \le 3.36 \abs*{z}^{-1/2} .\]
  The analogous estimates for the other variants of Bessel's function
  are clear.
\end{proof}

With Lemma~\ref{stlemma}, we estimate the function~\eqref{functionf}
on the strip $\Im\scp*{z}\le 1$ and
$\Re\scp*{z}\ge 0.95 N^{\frac 32 }\ln\scp*{N}$ in absolute value by
${\scp*{3.36}}^6 \abs*{z}^{-2} $. Cutting the interval $\scs*{S, T}$
into $K$ intervals of length $2d=\scp*{T-S}/K$ and using
Gauss-Legendre quadrature as in Lemma~\ref{abgauss} on each
interval, we obtain
\begin{equation}\label{eq:gauss-nueric-error-on-[S, T]}
  \begin{split}
    \abs*{I_{\bf k}^{S, T}-\tilde{I}_{\bf k}^{S, T}}
    &= \abs*{
      \int_{S}^T{f}\scp*{z}\id{z}
      -\sum_{j=0}^{K-1} d \sum_{i=1}^n{w}_i f(S+d\scp*{2j+1}+dx_i) }\\
    &\le \frac{\pi}{ 2^{2n}} d^{2n+1} {\scp*{3.36}}^6 \sum_{j=0}^{K-1} {\scp*{S+2dj-1}}^{-2}\\
    &\le \frac{\pi}{2^{2n+1}} d^{2n} {\scp*{3.36}}^6 \int_S^T {\scp*{x-1-2d}}^{-2} \id{x}\\
    &\le \pi 2^{-2n-1} d^{2n}{\scp*{3.36}}^6 {\scp*{S-1-2d}}^{-1}\\
    &\le 0.000135 \times {\scp*{S-1-2d}}^{-1} d^{24}\\
    &\le 1.1 \times 10^{-10}.
  \end{split}
\end{equation}
Here we used $n=12$ in the penultimate and $S=6{,}000$ and $d=0.8$ in
the ultimate inequality. The use of Lemma~\ref{stlemma} here requires
$S \ge 0.95 N^{3/2}\ln\scp*{N}+1$, which is satisfied with $N=120$ and
$S=6{,}000$.  Assuming $T=150{,}000$, this amounts to $1{,}080{,}000$
evaluations of the integrand.

\subsection{Asymptotic approximation on the interval  \texorpdfstring{$\left[T, \infty\right)$}{[T, infinity)}}

We present a precise error bound for the classical asymptotic
expansion of order four of Bessel functions, a slight refinement of
the corresponding discussion in Section 2 of~\cite{e2015estimates}.

\begin{lemma}
  Assume $N\ge 20$ and $n\le N$. Assume $z$ is real and
  $z> 2 N^{\frac 32} \ln\scp*{N} $.  Then
  \begin{equation}\label{4order}
    \begin{split}
  \scp*{\frac {\pi z}2}^{\frac 12}
  J_n^{++} \scp*{z}e^{-i\omega}
  &=
  1+ i \frac{\Gamma\scp*{\nu+2} }{2\Gamma\scp*{\nu}z} - \frac{\Gamma\scp*{\nu+3} }{8 \Gamma\scp*{\nu-1} z^2}\\
  &\qquad{}- i \frac{\Gamma\scp*{\nu+4}}{48 \Gamma\scp*{\nu-2}z^3}
  +R
  \end{split}
  \end{equation}
  with
  \[\abs*{R}\le 0.0043 \frac{N^8}{z^4}.\]
\end{lemma}

\begin{proof}
  Recalling~\eqref{jplusplus}, we see
  \begin{equation}\label{modj}
    {\scp*{\frac {\pi z}2}}^{\frac 12} J_n^{++} \scp*{z}e^{-i\omega}
    =\frac 1 {\Gamma\scp*{\nu+1}}
    \int_0^\infty e^{-u} u^{\nu} {\scp*{1+\frac {iu}{2z}}}^\nu \id{u}.
  \end{equation}
  We expand ${\scp*{1+ix}}^\nu$ for real $0\le x$ with Taylor's theorem
  into
  \[{\scp*{1+ix}}^\nu=1+i \nu x- \frac 12 \nu \scp*{\nu-1}x^2
    - \frac 1 6 i \nu \scp*{\nu-1}\scp*{\nu-2}x^3+ r\]
  where
  \begin{align*}
    \abs*{r}
    &=\abs*{
      \int_0^x \frac 1 {6} \nu \scp*{\nu-1}\scp*{\nu-2}\scp*{\nu-3}  {\scp*{x-t}}^3 {\scp*{1+it}}^{\nu-4}\, \id{t}}\\*
    &\le \frac 1 {24} \abs*{\nu \scp*{\nu-1}\scp*{\nu-2}\scp*{\nu-3}}  x^4 {\scp*{1+x^2}}^{\frac 12 \max\scb*{\nu-4, 0}}.
  \end{align*}
  Thus the right hand side of~\eqref{modj} becomes~\eqref{4order} with
  \[
    \abs*{R}\le\int_0^\infty e^{-u} u^\nu \abs*{r\scp*{u}}\id{u}
  \]
  and $r\scp*{u}$ similar to above with $x$ replaced by $u/2z$.  We cut
  the integral at $u=4N\ln\scp*{N}$. For $u\le 4N\ln\scp*{N}$ we have
  \[
    \abs*{r\scp*{u}}
    \le \frac{\abs{\nu \scp*{\nu-1}\scp*{\nu-2}\scp*{\nu-3}} u ^4 }{384 z^4}{\scp*{1+\frac 1N}}^{\frac N2}
  \]
  and thus
  \[
    \frac 1{\Gamma\scp*{\nu+1}}\int_0^{4N\ln\scp*{N}} e^u u^\nu r\scp*{u}\id{u}
    \le e^{1/2}
    \frac{\Gamma\scp*{\nu+5}}{384 \abs{\Gamma\scp*{\nu-3}} z^4}\le \frac {e^{1/2}
      N^8}{384 z^4}.
  \]
  For $u\ge 4N\ln \scp*{N}$ we have
  \[
    \abs*{r\scp*{u}}
    \le \frac{\abs{\nu \scp*{\nu-1}\scp*{\nu-2}\scp*{\nu-3}}  u ^4 }{6 {\scp*{2z}}^4}{\scp*{\frac uN}}^{N/2},
  \]
  and therefore
  \begin{align*}
  \MoveEqLeft
  \frac 1{\Gamma\scp*{\nu+1}} \int_{4N\ln\scp*{N}}^\infty e^{-u} u^\nu r\scp*{u}\id{u}
  \\ &\le
  \frac 1{384 z^4 \abs{\Gamma\scp*{\nu-3}}} \int_0^\infty e^{-u/2}e^{-2N\ln\scp*{N}} u^{\nu+4}
  {\scp*{\frac u N}}^{N/2}\id{u}
  \\ &\le
  \frac 1{96 z^4 \abs{\Gamma\scp*{\nu-3}}} \frac{2^{\nu+5+N/2}}{N^{2N+N/2}}\Gamma\scp*{\nu+5+N/2}
  \le z^{-4}.
  \end{align*}
    Adding the estimates for the two pieces of the integral gives the
  bound claimed in the lemma.
\end{proof}

With the above lemma, we obtain for $J^+_n$ in the range of $n$ and $z$
discussed in the lemma:
\begin{equation}\label{4terms}
  {\scp*{\frac {\pi z}2}}^{\frac 12}
  J_n^{+} \scp*{z} =a + bz^{-1}+ cz^{-2}+dz^{-3}+R
\end{equation}
with
\begin{align*}
a&=\cos\scp*{\omega},\\
b&=-  \sin\scp*{\omega}\frac 12 \scp*{n^2-\frac 14},\\
c&= - \cos\scp*{\omega}\frac 1 8 \scp*{n^2-\frac 14}\scp*{n^2-\frac 9 4},\\
d&= \sin\scp*{\omega} \frac 1{48} \scp*{n^2-\frac 14}\scp*{n^2-\frac 9 4}\scp*{n^2-\frac {25}4},
\end{align*}
and $R$ satisfies the same bounds as in the lemma. The analogous
identity holds for $J_n$ since it coincides with $J_n^+$ on the real
line.  Now we consider six indices $n_j$ and corresponding $\omega_j \vcentcolon= z - \pi/4 - n_j\pi/2$
and obtain for real $z$
\[
{\scp*{\frac {\pi} 2}}^{3}  \prod_{j=1}^6 J_{n_j} \scp*{z} z
=Az^{-2}+Bz^{-3}+Cz^{-4}+Q
\]
with
\begin{align*}
  A &= \cos\scp*{\omega_1}\cos\scp*{\omega_2}\cos\scp*{\omega_3}\cos\scp*{\omega_4}\cos\scp*{\omega_5}\cos\scp*{\omega_6},\\
  B &= - \sum_{j'=1}^{6} \frac{n_{j'}^2-1/4}2
      \sin\scp*{\omega_{j'}} \prod_{1 \leq j \leq 6, j \neq j'} \cos\scp*{\omega_j},\\
  C &= \sum_{j'=1}^{6} \sum_{j''=j'+1}^{6} \frac{\scp*{n_{j'}^2-1/4}\scp*{n_{j''}^2-1/4}}{4}
  \sin\scp*{\omega_{j'}}\sin\scp*{\omega_{j''}}
  \prod_{1 \leq j \leq 6, j \neq j', j\neq j''} \cos\scp*{\omega_{j}}\\
  &\qquad{} - \frac {\scs*{\sum_{j=1}^6 \scp*{n_j^2-1/4}\scp*{n_j^2-9/4}}}8
      \prod_{j=1}^6\cos\scp*{\omega_j}.
\end{align*}

The remainder term $Q$ we estimate from above. For this we collect
terms of order $z^{-5}$, $z^{-6}$, $z^{-7}$, $z^{-8}$ and higher order
separately. We begin with a remark on integrals of type
\[\int_T^\infty \prod_{j=1}^6 \phi_j\scp*{\omega_j}z^{-5}\]
where an odd number of functions $\phi_j$ are the cosine function and
an odd number of $\phi_j$ are the sine function. If a function is odd
in the variable $z$ about the point $\frac \pi 4$ then we call it of
parity $-1$, and if it is even we call it of parity $1$.  The function
\[\sin\scp*{\omega_j}=\sin\scp*{z-\frac \pi 4 - \frac {n\pi}{2}}\]
has parity $-{\scp*{-1}}^{n}$ about the point $\frac \pi 4$, while the
function
\[\cos\scp*{\omega_j}=\cos\scp*{z-\frac \pi 4 - \frac {n\pi}{2}}\]
has parity ${\scp*{-1}}^n$.  A product of six such functions with
$\sum_{j=1}^6 n_j$ even and an involving an odd number of sine function is
therefore odd about the point $\pi/4$.  This means it integrates to
zero about each period of the periodic function.  Hence by partial
integration
\begin{align*}
  \abs*{\int_T^\infty \prod_{j=1}^6 \phi_j\scp*{\omega_j}z^{-5}\id{z}}
  &\le \abs*{\int_T^\infty \partial_{z}^{-1} \scs*{\prod_{j=1}^6 \phi_j\scp*{\omega_j}} 5z^{-6}\id{z}}+T^{-5}\\
  &\le \scp*{\pi+1} T^{-5}
\end{align*}
where we used that the function $\prod_{j=1}^6 \phi_j\scp*{\omega_j}$ is
bounded by $1$ and since it integrates to $0$ of periods of length
$2\pi$ its primitive $\partial_{z}^{-1} \scs*{\prod_{j=1}^6 \phi_j\scp*{\omega_j}}$
is bounded by $\pi$.

We use this estimate in each of the terms of the fifth order, all of
which have an odd number of sine functions. We estimate all factors
$n_j^2-x$ by $N^2$. Counting the terms and referring to the
abbreviations in~\eqref{4terms}, we obtain $6$ terms with a factor
$d$, $30$ factors with a product $bc$ and $20$ terms with a factor
$b^3$. Thus this term is estimated by
\begin{equation}\label{fifth}
  \scp*{\frac 6{48}+\frac {30}{16}+\frac {20}{8}} \scp*{1+\pi}\frac {N^6}{T^5} \le 19  \frac {N^6}{T^5}.
\end{equation}

To estimate sixth order terms we estimate all sine and cosine
functions by $1$.  Integrating $z^{-6}$ then simply gives $T^{-5}/5$.
We obtain $6$ terms with a factor $R$, $30$ terms with a factor $bd$,
$15$ terms with a factor $c^2$, $60$ terms with a factor $cb^2$, $15$
terms with a factor $b^4$. This gives the estimate
\begin{equation}\label{sixth}
\scp*{ 6 \times 0.0043 + \frac {30}{96}+\frac {15}{64}+\frac{60}{32}+\frac {15}{16}}
\frac{N^8}{5T^5}\le 0.68\frac{N^8}{T^5}.
\end{equation}

The seventh order terms we estimate similarly.  We obtain $30$ terms
with a factor $Rb$, $30$ terms with a factor $dc$, $60$ terms with a
factor $db^2$, $60$ terms with a factor $c^2b$, $60$ terms with a
factor $cb^3$ and $6$ terms with a factor $b^5$. This gives the
estimate
\begin{equation}\label{seventh}
  \scp*{
    \frac {30}2 \times 0.0043
    + \frac {30}{384}
    + \frac {60}{192}
    + \frac{60}{128}
    + \frac{60}{64}
    + \frac {6}{32}
  }
  \frac{N^{10}}{6T^6}\le 0.35\frac{N^{10}}{T^6}.
\end{equation}

Counting the eighth order terms we find $30$ terms with a factor $Rc$,
$60$ terms with a factor $Rb^2$, $15$ terms with a factor $d^2$, $120$
terms with a factor $dcb$, $60$ terms with a factor $db^3$, $15$ terms
with a factor $c^3$, $45$ terms with a factor $c^2b^2$, $30$ terms
with a factor $cb^4$, and one term with a factor $b^6$. This gives the
estimate
\begin{equation}\label{eighth}
  \begin{split}
  \left\lparen
  \vphantom{0.0043\scp*{\frac {30}8 + \frac {60}4} + \frac {15}{2304} +\frac {120}{768} +\frac{60}{384} +\frac {15}{512} +\frac {45}{256} +\frac {30}{128} + \frac {1}{32}}
  0.0043
    \scp*{
      \frac {30}8
      + \frac {60}4}
  \right.
    &+ \frac {15}{2304}
    +\frac {120}{768}
    +\frac{60}{384}
  \\
  &
  \left.
    \vphantom{0.0043\scp*{\frac {30}8 + \frac {60}4} + \frac {15}{2304} +\frac {120}{768} +\frac{60}{384} +\frac {15}{512} +\frac {45}{256} +\frac {30}{128} + \frac {1}{32}}
    +\frac {15}{512}
    +\frac {45}{256}
    +\frac {30}{128}
    +\frac {1}{32}
  \right\rparen
  \frac{N^{12}}{7 T^7}
  \le 0.13\frac{N^{12}}{T^7}.
\end{split}
\end{equation}

The terms of order $9$ or higher we estimate more crudely. There are
at most $5^6$ terms, the product of pre factors being at most $0.0043$
if a factor $R$ is involved and being at most $\frac 1{128}$ if no
such factor is involved. Thus we get, assuming $N^2\le T$, the upper bound
\begin{equation}\label{ninth}
  \le \frac {5^6}{128}\frac{N^{14}}{8T^8}\le 16 \frac{N^{14}}{T^8}.
\end{equation}

Assuming $N=120 \ge 20$ and $T\ge 10N^2$ we may
add~\eqref{fifth},~\eqref{sixth},~\eqref{seventh},~\eqref{eighth},
and~\eqref{ninth} to
\begin{equation}\label{eq:tail-error}
  \begin{split}
    \abs*{I_{\bf k}^{T, \infty}-\tilde{I}_{\bf k}^{T, \infty}}
    &\le
    \left\lparen
      \vphantom{19\times  10^{-4}20^{-2}+0.68\times 10^{-4} + 0.35\times 10^{-5}+0.13\times 10^{-6}+ 16\times 10^{-7}}
      19\times  10^{-4}20^{-2}+0.68\times 10^{-4} + 0.35\times 10^{-5}
    \right. \\
    &\qquad
    \left.
      \vphantom{19\times  10^{-4}20^{-2}+0.68\times 10^{-4} + 0.35\times 10^{-5}+0.13\times 10^{-6}+ 16\times 10^{-7}}
      +0.13\times 10^{-6}+ 16\times 10^{-7}
    \right\rparen{} T^{-1}\\
    &\le 0.75\times 10^ {-4}T^{-1}\\
    &\le 0.5 \times 10^{-9}.
  \end{split}
\end{equation}
Here we used $T=150{,}000 > 10N^2$ in the last
inequality. Summing~\eqref{eq:gauss-nueric-error-on-[0,
  S]},~\eqref{eq:gauss-nueric-error-on-[S, T]},
and~\eqref{eq:tail-error}, we finally obtain
\begin{equation}\label{eq:approx-error}
  \abs*{I_{\bf k}-\tilde{I}_{\bf k}}\le 0.73 \times 10^{-9}.
\end{equation}

\section{Numerical Results}%
\label{numresults}

Using the approximations $\tilde{I}_{\bf k}$, one computes
approximations
${\scp{\tilde{Q}_{{\bf m}, {\bf n}}}}_{{\bf m}, {\bf n}\in {X}_{D}}$
to the matrices
${\scp{Q_{{\bf m}, {\bf n}}}}_{{\bf m}, {\bf n}\in {X}_{D}}$
analogously to the formulae in the introduction.  Using evaluations of
Bessel functions of sufficient accuracy and arbitrary-precision
arithmetic, the $\tilde{I}_{\bf k}$ and $\tilde{Q}_{{\bf m}, {\bf n}}$
as well as the eigenvalues of the matrices
${\scp{\tilde{Q}_{{\bf m}, {\bf n}}}}_{{\bf m}, {\bf n}\in {X}_{D}}$
were computed up to an error of at most $10^{-20}$.  The computed
smallest eigenvalue for each $D$ is plotted in
Figure~\ref{fig:smallest_evs} and shown in
Table~\ref{tab:smallest_evs_zero_to_ten} for small $D$.  In
particular, the matrix
${\scp{\tilde{Q}_{{\bf m}, {\bf n}}}}_{{\bf m}, {\bf n}\in {X}_{D}}$
is positive definite with smallest eigenvalue at least $0.000036$.

\begin{table}
\begin{center}
\begin{tabular}{cc}
  \toprule
  $D$&$\lambda_{\min}$\\
  \midrule
  0  & 
       0.0000369980\\
  2  & 
       0.0000371564\\
  4  & 
       0.0000374854\\
  6  & 
       0.0000379002\\
  8  & 
       0.0000384081\\
  10 & 
       0.0000389622\\
  \bottomrule
\end{tabular}
\end{center}
\caption{Smallest eigenvalues for matrices
  ${\scp*{Q_{{\bf m}, {\bf n}}}}_{{\bf m}, {\bf n}\in X_{D}}$, for
  $D = 0, 2, \ldots, 10$, rounded to 11 significant figures.}%
\label{tab:smallest_evs_zero_to_ten}
\end{table}

We have for each ${\bf m}, {\bf n}$ in question
from~\eqref{eq:approx-error}
\[\abs*{{Q}_{{\bf m}, {\bf n}} - \tilde{Q}_{{\bf m}, {\bf n}}}
\le 16 \times 0.73 \times 10^{-9}\le 1.2 \times 10^{-8} .\] This
entry-wise bound is multiplied by the size $1860$ of the set $X_D$ for
$N=120$ to obtain a bound for the operator norm

\[\norm*{
{\scp*{{Q}_{{\bf m}, {\bf n}}}}_{{\bf m}, {\bf n}\in {X}_{D}}- {\scp*{\tilde{Q}_{{\bf
    m}, {\bf n}}}}_{{\bf m}, {\bf n}\in {X}_{D}} }_{op}\le 1860 \times
1.2 \times 10^{-8} \le 0.000023 .\] As this is less than the smallest
eigenvalue of the matrix
\[{\scp*{\tilde{Q}_{{\bf m}, {\bf n}}}}_{{\bf m}, {\bf n}\in {X}_{D}},\]
we may deduce that the matrix
${\scp{{Q}_{{\bf m}, {\bf n}}}}_{{\bf m}, {\bf n}\in {X}_{D}}$ is positive
definite as well.  Following the reductions
of~\cite{oliveira2019band}, this proves Theorem~\ref{thm:main}.

While the error bounds~\eqref{eq:approx-error} are good enough to
prove the main theorem, they are not good enough to guarantee that our
plots and tables of eigenvalues of
${\scp{\tilde{Q}_{{\bf m}, {\bf n}}}}_{{\bf m}, {\bf n}\in {X}_{D}}$ are
representative for those of
${\scp{{Q}_{{\bf m}, {\bf n}}}}_{{\bf m}, {\bf n}\in {X}_{D}}$ within the
margins suggested by the visible information.  However, the arguments
obtaining~\eqref{eq:approx-error} as well as the use of the operator
norm above are very crude estimates. With very high likelihood, the
true differences between the corresponding eigenvalues of the two
matrices are much smaller than the operator norm above.  In this
sense, we consider our plots and tables of eigenvalues of
${\scp{\tilde{Q}_{{\bf m}, {\bf n}}}}_{{\bf m}, {\bf n}\in {X}_{D}}$ as very
much representative of those of
${\scp{{Q}_{{\bf m}, {\bf n}}}}_{{\bf m}, {\bf n}\in {X}_{D}}$.  This is in
accordance with the rather regular behaviour of the plots and of
Table~\ref{tab:smallest_evs_zero_to_ten}, a level of regularity that
one might expect from the eigenvalues of
${\scp{{Q}_{{\bf m}, {\bf n}}}}_{{\bf m}, {\bf n}\in {X}_{D}}$.

A number of numerical findings would be interesting to understand
analytically and maybe prove for large $N$ asymptotically. Any
progress in this understanding would presumably require an asymptotic
understanding of the entries of the matrices in~\eqref{qmatrices}.
Figure~\ref{fig:column} shows sample columns of these matrices. For
nicer visualization, we have undone the dimension reduction by
symmetry in~\cite{oliveira2019band} and shown the related matrices
\begin{equation}\label{qtmatrices}
\tilde{Q}^{\scp*{D}} \vcentcolon= {\scp*{Q_{{\bf m}, {\bf n}}}}_{{\bf m}, {\bf n}\in Z_{D}},
\end{equation}
in the space
\[
Z_{D} \vcentcolon=
\Set*{ \scp*{m_1, m_2, m_3}\in {\scp*{2\mathbb{Z}}}^3\setminus \{\scp*{0, 0, 0}\}
  \given \substack{m_1+m_2+m_3=D,\\
    \abs*{m_1}, \abs*{m_2}, \abs*{m_3}\le N}}.
\]
The space $Z_{D}$ is naturally depicted as a hexagon. It has a sixfold
symmetry under permutations of the three elements, which is the
symmetry group of a regular hexagon. This symmetry extends to
symmetries of the columns shown and is visible in the plots of
Figure~\ref{fig:column}. The space $X_D$ in our previous calculations
is only one fundamental domain under this symmetry.

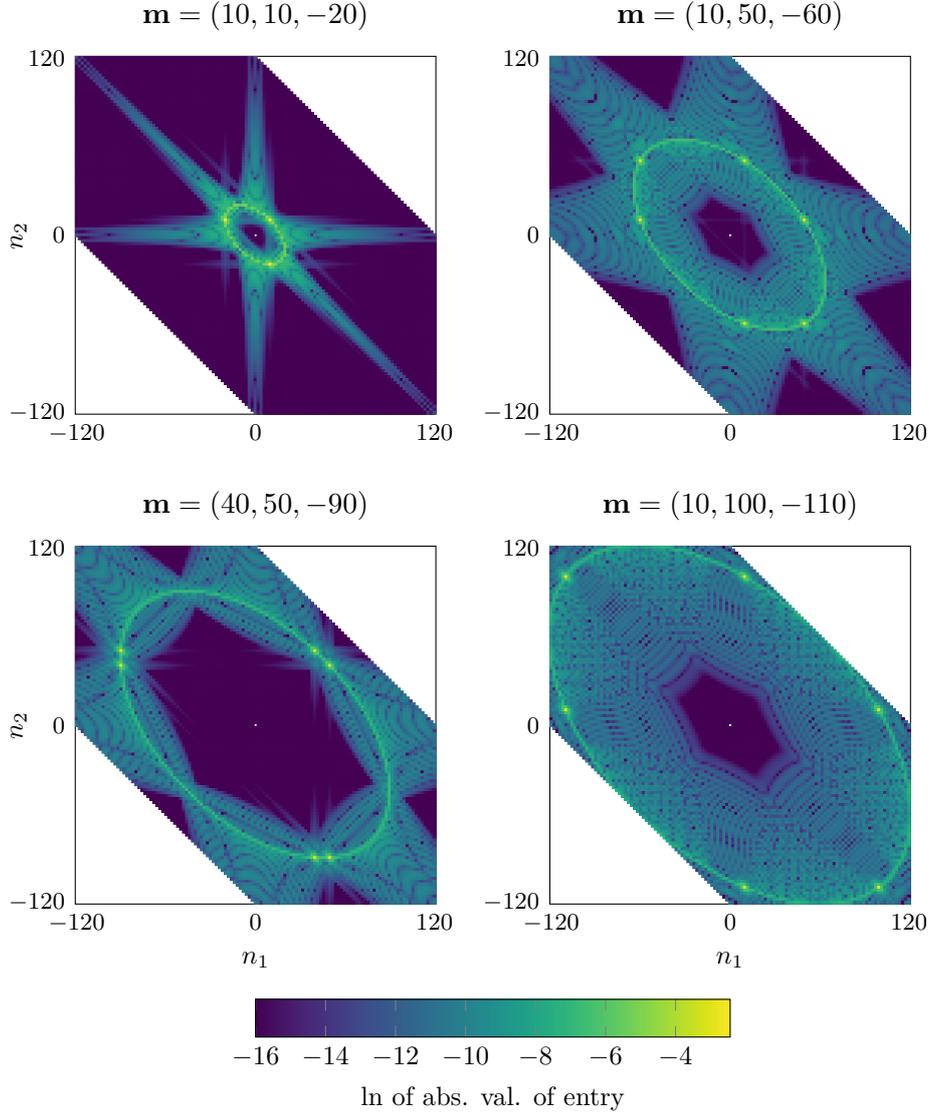
\begin{figure}
  {\centering
    \begin{tikzpicture}
    \begin{groupplot}[
      group style={
        group size=2 by 2,
        vertical sep=1.75cm,
        horizontal sep=1.5cm,
        xlabels at=edge bottom,
        ylabels at=edge left,
      },
      colormap/viridis,
      enlargelimits=false,
      scale only axis=true,
      width=0.375\textwidth,
      height=0.375\textwidth,
      small,
      xtick={-120, 0, 120},
      ytick={-120, 0, 120},
      tick style={draw=none},
      xlabel=$n_1$,
      ylabel=$n_2$,
      y label style={yshift=-0.5cm},
      point meta min=-16,
      point meta max=-2.44064873320304770728701,
      ]

      \nextgroupplot[
      title={${\bf{m}}=\scp*{10, 10, -20}$},
      width=0.375\textwidth,
      height=0.375\textwidth,
      scale only axis=true,
      ]
      \addplot[
      matrix plot*,
      point meta=explicit
      ] table[meta index=2] {figure_data/figure_4_col_10_10.dat};

      \nextgroupplot[
      title={${\bf{m}}=\scp*{10, 50, -60}$},
      width=0.375\textwidth,
      height=0.375\textwidth,
      scale only axis=true,
      ]
      \addplot[
      matrix plot*,
      point meta=explicit
      ] table[meta index=2] {figure_data/figure_4_col_10_50.dat};

      \nextgroupplot[
      title={${\bf{m}}=\scp*{40, 50, -90}$},
      width=0.375\textwidth,
      height=0.375\textwidth,
      scale only axis=true,
      colorbar horizontal,
      colorbar style={
        at={(parent axis.below south)},
        anchor=north west,
        width=\pgfkeysvalueof{/pgfplots/parent axis width} +  \pgfkeysvalueof{/pgfplots/group/horizontal sep},
        xlabel={$\ln$ of abs. val. of entry},
        tick style={draw=gray},
        label style={font=\small},
        tick label style={font=\small},
        ytick distance=1,
      }
      ]
      \addplot[
      matrix plot*,
      point meta=explicit
      ] table[meta index=2] {figure_data/figure_4_col_40_50.dat};

      \nextgroupplot[
      title={${\bf{m}}=\scp*{10, 100, -110}$},
      width=0.375\textwidth,
      height=0.375\textwidth,
      scale only axis=true,
      ]
      \addplot[
      matrix plot*,
      point meta=explicit
      ] table[meta index=2] {figure_data/figure_4_col_10_100.dat};
    \end{groupplot}
  \end{tikzpicture}
}
\caption{Columns of~\eqref{qtmatrices} with $D=0$ and $N=120$ for
  sample fixed values of ${\bf m}$, in coordinates $n_{1},
  n_{2}$. Coloring is logarithmic in the absolute values of entries;
  logarithmic values below $-16$ are clipped.}%
\label{fig:column}
\end{figure}

The largest entry in each column of the matrices in
$\eqref{qmatrices}$ is the diagonal element. The diagonal elements appear as brights dots on the ellipses in Figure~\ref{fig:column}. Due to the sixfold
symmetry of the visualization, they appear three or six times in each
image depending on their orbit under the symmetry group, namely the
points where $\mathbf{m}$ is a permutation of $\mathbf{n}$.  While it
may be tempting to try to prove positive definiteness of the matrices
in~\eqref{qmatrices} by diagonal dominance, for $N=120$ and $D=0$, the
ratios
\begin{equation}\label{eq:diagonal-non-dominance}
r_{\bfm} \vcentcolon= \abs{Q_{\bf{m}, \bf{m}}}^{-1} \sum_{\bfn \in X_{0}\setminus\Set{\bfm}} \abs{Q_{\bf{m}, \bf{n}}},
\end{equation}
can be large as well as small, e.g.
\[
r_{\scp*{-2, 0, 2}} \approx 3.1,
\quad
r_{\scp*{-90, 40, 50}} \approx 1.9,
\quad
r_{\scp*{-4, 2, 2}} \approx 0.7.
\]
Indeed, this failure of this attempt is natural due to the existence
of some very small eigenvalues.

A secondary collection of large entries in each column shown in
Figure~\ref{fig:column} can be seen as yellow ellipses in the
diagrams.  These visual ellipses correspond to circles in a
visualization of the domain as regular hexagon.  The circles were
already observed in~\cite{oliveira2019band}. The current data even
more strongly suggests that these secondary peaks are located on the
surface
\begin{equation}
\label{eq:observed-ellipse}
n_{1}^{2}+n_{2}^{2}+n_{3}^{2} = m_{1}^{2}+m_{2}^{2}+m_{3}^{2}.
\end{equation}
Indeed, in Figure~\ref{fig:column-fit}, the ellipse/circle given
by~\eqref{eq:observed-ellipse}, plotted as a white line, is overlayed
onto one of the plots, showing that it matches the yellow peaks very
well. The size of the sum of off-diagonal elements
in~\eqref{eq:diagonal-non-dominance} is essentially entirely due to
the elements near the circle.

\begin{figure}
  \begin{tikzpicture}
    \begin{axis}[
      colormap/viridis,
      enlargelimits=false,
      scale only axis,
      width=0.6\textwidth,
      height=0.6\textwidth,
      xtick={-120, 0, 120},
      ytick={-120, 0, 120},
      tick style={draw=none},
      xlabel=$n_1$,
      ylabel=$n_2$,
      colorbar,
      colorbar style={
        ylabel={$\ln$ of abs. val. of entry},
        tick style={draw=gray},
        label style={font=\small},
        tick label style={font=\small},
        title style={font=\small},
      },
      point meta min=-16,
      point meta max=-3.11999,
      ]

      \addplot[
      matrix plot*,
      point meta=explicit
      ] table[meta index=2] {figure_data/figure_5_col_20_70.dat};
      \addplot[
      domain=-pi:pi,
      samples=300,
      dashed,
      white,
      line width=0.25pt]
      (
      {(10/3)*sqrt(67)*(3*cos(2*pi*deg(x)) + sqrt(3)*sin(2*pi*deg(x)))},
      {(10/3)*sqrt(67)*(-3*cos(2*pi*deg(x)) + sqrt(3)*sin(2*pi*deg(x)))}
      );
    \end{axis}
  \end{tikzpicture}
  \caption{Column ${\bf m} = \scp*{20, 70, -90}$
    of~\eqref{qtmatrices} with $D=0$ and $N=120$, in coordinates
    $n_{1}, n_{2}$. Coloring is logarithmic in the absolute values of
    entries; logarithmic values below $-16$ are clipped. The white
    line indicates the ellipse given by~\eqref{eq:observed-ellipse}.}%
\label{fig:column-fit}
\end{figure}
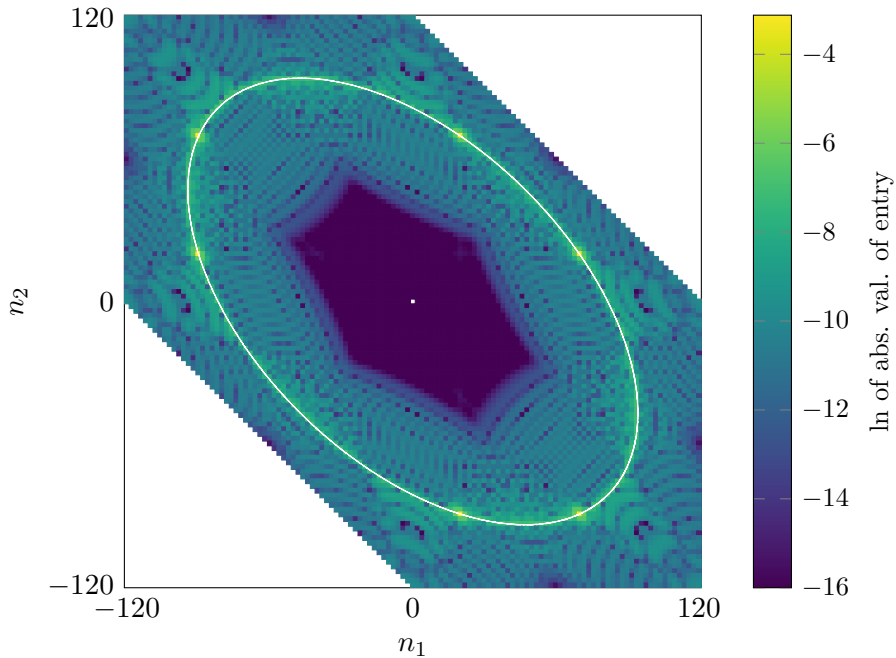

At the moment, we do not even have a qualititive understanding of the
occurrence of these circles, let alone a quantitative one, which would
probably be required if one wanted to extract positive definiteness
from the structure of these matrices.  Figure~\ref{fig:circle-radial}
plots the values of the entries of the column $\scp*{40, 60, -100}$ as a
function of the radial variable $\sqrt{n_{1}^{2}+n_{2}^{2}+n_{3}^{2}}$
in the vicinity of the radius of the yellow circle, which is about
$123.3$. The plot strongly suggests that the pattern of the circle is
asymptotically well approximated by a smooth radial function of low
complexity.  The value at the diagonal element, which is at radius
about $123.3$, is not plotted in Figure~\ref{fig:circle-radial},
because it is too large at about $0.044$.  Likewise, two further
entries near the diagonal elements are not shown, they have value
$-0.0024$. They are part of the small ring of six large elements
around the diagonal elements, which also include the four entries with
values near $-0.0015$ shown in Figure~\ref{fig:circle-radial}.  Also,
small matrix entries are cut off in the plot~\ref{fig:circle-radial}.

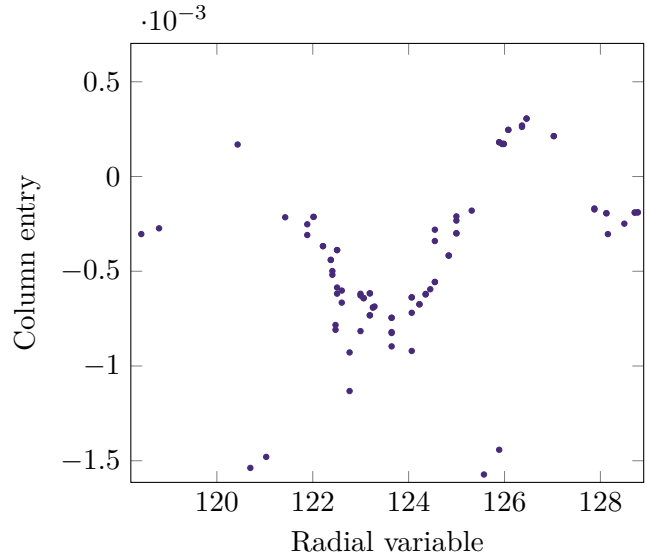
\begin{figure}
  {\centering
    \begin{tikzpicture}
      \begin{axis}[
        scale only axis,
        width=0.8*\axisdefaultwidth,
        xlabel=Radial variable,
        ylabel=Column entry,
        xmin=118.205,
        xmax=128.903,
        ymin=-0.00161428,
        ymax=0.000702375,
        cycle list={[indices of colormap={2} of colormap/viridis]},
        ]
        \addplot+ [only marks, mark size=1pt] table {figure_data/figure_6.dat};
      \end{axis}
    \end{tikzpicture}
  }
  \caption{Partial plot of the values of the entries of column
    ${\bf m} = \scp*{40, 60, -100}$ of~\eqref{qtmatrices} with $D=0$
    and $N=120$ in dependence of the radial
    variable.}\label{fig:circle-radial}
\end{figure}

Due to the monotonicity shown in Figure~\ref{fig:smallest_evs}, it is
of particular interest to study the matrix in~\eqref{qmatrices} for
$D=0$. Figure~\ref{fig:evs_D_0} shows all eigenvalues of the matrix
$D=0$ for $N=120$ sorted and enumerated by size. For comparison, we
also show the analogous plot for the matrix $D=200$, a diagram that is
somewhat similar.  Note the two jumps of the diagram for $D=0$ at
about the $60$-th smallest and $60$-th largest eigenvalues.  An
analytical proof of positivity for all $N$ in the asymptotic regime
would require a better understanding of the ensemble of small
eigenvalues below the first jump.

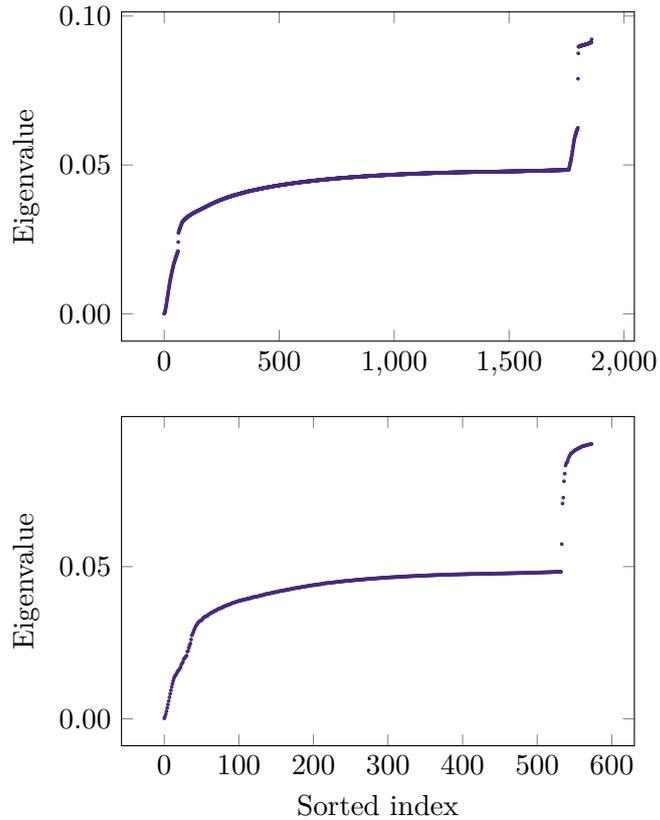
\begin{figure}
  {\centering
    \begin{tikzpicture}
      \begin{groupplot}[
        group style={
          group size=1 by 2,
          xlabels at=edge bottom,
          ylabels at=edge left,
        },
        cycle list={[indices of colormap={2} of colormap/viridis]},
        xlabel=Sorted index,
        ylabel=Eigenvalue,
        ]

        \nextgroupplot[
        scale only axis,
        width=0.8*\axisdefaultwidth,
        height=0.6*\axisdefaultheight,
        scaled ticks=false,
        y tick label style={
            /pgf/number format/.cd,
            fixed,
            fixed zerofill,
            precision=2,
            /tikz/.cd
          }]
        \addplot+ [only marks, mark size=0.5pt] table {figure_data/figure_3_d_0.dat};

        \nextgroupplot[
        scale only axis,
        width=0.8*\axisdefaultwidth,
        height=0.6*\axisdefaultheight,
        scaled ticks=false,
        y tick label style={
            /pgf/number format/.cd,
            fixed,
            fixed zerofill,
            precision=2,
            /tikz/.cd
          }]
        \addplot+ [only marks, mark size=0.5pt] table {figure_data/figure_3_d_200.dat};
      \end{groupplot}
    \end{tikzpicture}
  }
  \caption{Sorted eigenvalues of~\ref{qmatrices} with $D=0$ (top, 1860
    eigenvalues) and $D=200$ (bottom, 574 eigenvalues), both with
    $N=120$.}%
  \label{fig:evs_D_0}
\end{figure}

The eigenfunctions corresponding to the smallest eigenvalues
experimentally appear to be essentially radial functions in the
visualization corresponding to Figure~\ref{fig:column}, cf.
Figure~\ref{fig:eigenvector-OTZ}. There are three notable observations
to be made here:
\begin{enumerate}
\item The bulk of the eigenvector depends essentially only on
  $\sqrt{n_{1}^{2}+n_{2}^{2}+n_{3}^{2}}$.
\item The value at the points where two of the $n_{i}$s coincide is
  smaller than suggested by this radial dependence by a factor of two
  (although this is not evident from this particular visualization).
\item Outside of the largest circle that fits into the region $Z_{0}$,
  the eigenvector is small.
\end{enumerate}
The second point above is related to the fact that in the
matrices~\eqref{qtmatrices}, the rows and the columns
of~\eqref{qtmatrices} appear as many times as there are distinct
permutations of $\bfm$ and $\bfn$ respectively, i.e. 3 and 6 times. It
therefore appears more natural to analyze the eigenvalues
of~\eqref{qtmatrices}, or, equivalently, the matrices given by
${\scp*{p_{\bfm}Q_{\bfm, \bfn}p_{\bfn}}}_{\bfm, \bfn \in X_{0}}$,
where $p_{\bfm}$ is the number of distinct permutations of $\bfm$.

\begin{figure}
  {\centering
    \begin{tikzpicture}
    \begin{axis}[
      colormap/viridis,
      enlargelimits=false,
      scale only axis,
      width=0.8*\axisdefaultwidth,
      height=0.8*\axisdefaultwidth,
      xtick={-120, 0, 120},
      ytick={-120, 0, 120},
      tick style={draw=none},
      xlabel=$n_1$,
      ylabel=$n_2$,
      colorbar,
      colorbar style={
        ylabel={entry value},
        tick style={draw=gray},
        label style={font=\small},
        tick label style={font=\small},
        title style={font=\small},
      }
      ]

      \addplot[
      matrix plot*,
      point meta=explicit
      ] table[meta index=2] {figure_data/figure_7.dat};
    \end{axis}
  \end{tikzpicture}
}
\caption{Eigenvector of~\eqref{qmatrices} with $D=0$ and $N=120$
  corresponding to the smallest eigenvalue in coordinates
  $n_{1}, n_{2}$. The dotted lines are points where the orbit of the
  symmetry group has only three elements. They are an artifact of
  dimension reduction: the values of the eigenvector at these points
  are half as large as those of the corresponding eigenvector of
  \eqref{qmatrices}.}%
\label{fig:eigenvector-OTZ}
\end{figure}
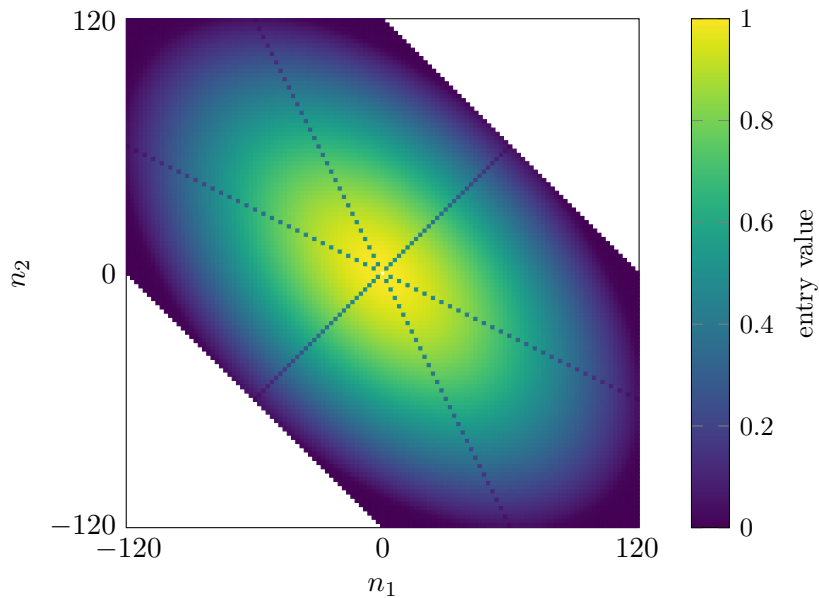

In view of the third point above, it also appears natural to truncate
the matrices not according to
$\max\scp*{\abs{n_{1}}, \abs{n_{2}}, \abs{n_{3}}}$, but according to
$\sqrt{n_{1}^{2}+n_{2}^{2}+n_{3}^{2}}$.  In what follows, let us
therefore consider the matrices
\begin{equation}
\label{l2matrices}
Q^{\circ} \vcentcolon= \scp*{p_{\bfm}Q_{\bfm, \bfn}p_{\bfn}}_{\bfm, \bfn \in X^{\circ}},
\end{equation}
where
\[
X^{\circ} =
\Set*{ \scp*{m_1, m_2, m_3}\in {\scp*{2\mathbb{Z}}}^3\setminus \{\scp*{0, 0, 0}\}
  \given \substack{m_1+m_2+m_3=0,\\
    \sqrt{m_{1}^{2}+m_{2}^{2}+m_{3}^{2}} \le \sqrt{3/2}N,\\
    m_{1} \leq m_{2} \leq m_{3} }}.
\]
is the largest disc contained in $X_{0}$.
The eigenvectors of the matrices~\eqref{l2matrices} corresponding to
small eigenvalues seem to be smooth functions of the radial variable
$\sqrt{n_{1}^{2}+n_{2}^{2}+n_{3}^{2}}$, see
Figure~\ref{fig:5eigenvectors-l2-truncation}.
In that figure, the five smallest eigenvalues are represented by different colors.
For each eigenvalue, there is a point for every $\bfn \in X^{\circ}$.
Surprisingly, the corresponding eigenvector entries $s_{\bf{n}}$ fall on a one-dimensional curve, although $X^{\circ}$ is taken from a two-dimensional lattice.
After a suitable rescaling, the profile of the curve seems to be independent of $N$,
as shown in Figure~\ref{fig:eigenvectors-l2-truncation-N-varies} for the smallest eigenvalue.

\begin{figure}
  {\centering
    \begin{tikzpicture}
      \begin{axis}[
        scale only axis,
        width=0.8*\axisdefaultwidth,
        xlabel=$\norm*{\bf{n}}_2$,
        ylabel=$s_{\bf{n}}$,
        legend cell align={left},
        legend image post style={scale=3, fill opacity=1.0},
        cycle list={[indices of colormap={0,3,6,9,12} of colormap/viridis]},
        ]

        \addplot+ [only marks, mark size=0.5pt, mark=*]
        table {figure_data/figure_8_sn_1.dat};
        \addlegendentry{$\lambda_1$}

        \addplot+ [only marks, mark size=0.5pt, mark=*]
        table {figure_data/figure_8_sn_2.dat};
        \addlegendentry{$\lambda_2$}

        \addplot+ [only marks, mark size=0.5pt, mark=*]
        table {figure_data/figure_8_sn_3.dat};
        \addlegendentry{$\lambda_3$}

        \addplot+ [only marks, mark size=0.5pt, mark=*]
        table {figure_data/figure_8_sn_4.dat};
        \addlegendentry{$\lambda_4$}

        \addplot+ [only marks, mark size=0.5pt, mark=*]
        table {figure_data/figure_8_sn_5.dat};
        \addlegendentry{$\lambda_5$}
      \end{axis}
    \end{tikzpicture}
  }
  \caption{Eigenfunctions associated to five smallest eigenvalues
    $\lambda_i, 1 \leq i \leq 5$ of~\eqref{l2matrices} with $D=0$ and
    $N=120$.}%
  \label{fig:5eigenvectors-l2-truncation}
\end{figure}
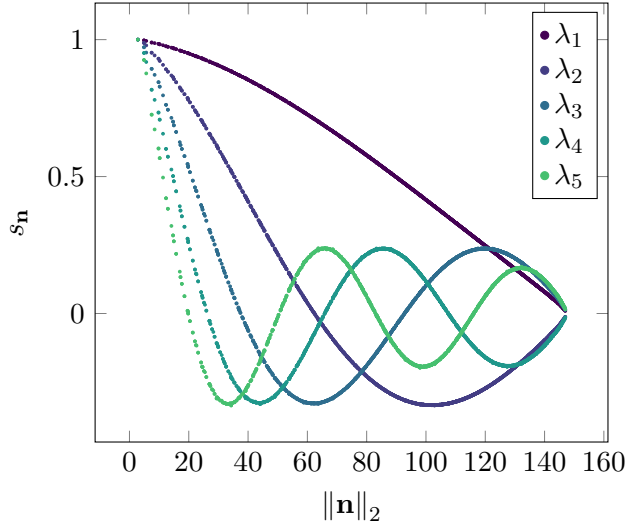

\begin{figure}
  {\centering
    \begin{tikzpicture}
      \begin{axis}[
        scale only axis,
        width=0.8*\axisdefaultwidth,
        xlabel=$\frac{\sqrt{n_1^2+n_2^2+n_3^2}}{\sqrt{\frac{3}{2}}N}$,
        ylabel=$s_{\bf{n}}$,
        legend cell align={left},
        legend image post style={scale=3, fill opacity=1.0},
        cycle list={[indices of colormap={0,3,6,9,12} of colormap/viridis]},
        ]

        \addplot+ [only marks, mark size=0.5pt, mark=*, fill opacity=0.75, draw opacity=0]
        table {figure_data/figure_9_n_40.dat};
        \addlegendentry{$N=40$}

        \addplot+ [only marks, mark size=0.5pt, mark=*, fill opacity=0.75, draw opacity=0]
        table {figure_data/figure_9_n_60.dat};
        \addlegendentry{$N=60$}

        \addplot+ [only marks, mark size=0.5pt, mark=*, fill opacity=0.75, draw opacity=0]
        table {figure_data/figure_9_n_80.dat};
        \addlegendentry{$N=80$}

        \addplot+ [only marks, mark size=0.5pt, mark=*, fill opacity=0.75, draw opacity=0]
        table {figure_data/figure_9_n_100.dat};
        \addlegendentry{$N=100$}

        \addplot+ [only marks, mark size=0.5pt, mark=*, fill opacity=0.75, draw opacity=0]
        table {figure_data/figure_9_n_120.dat};
        \addlegendentry{$N=120$}
      \end{axis}
    \end{tikzpicture}
  }
  \caption{Eigenfunctions associated to the smallest
    eigenvalues of~\eqref{l2matrices} with $D=0$ and
    $N\in\Set{40, 60, 80, 100, 120}$.}%
  \label{fig:eigenvectors-l2-truncation-N-varies}
\end{figure}
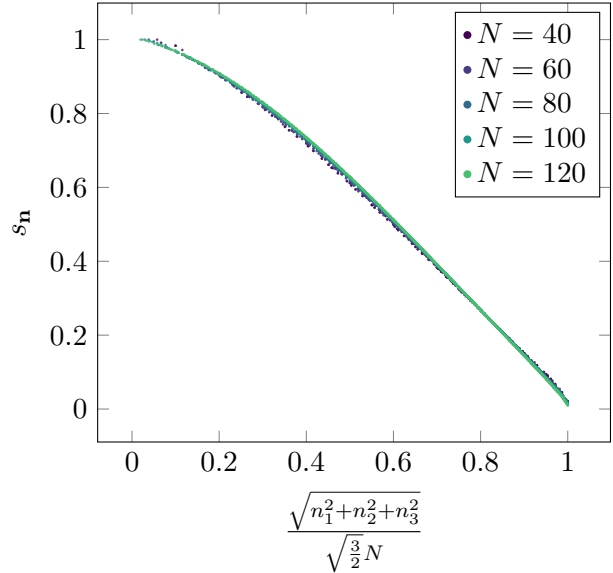

It is natural to link the behaviour of the eigenfunctions to the
smallest eigenvectors to a natural enemy of the sharp Fourier
extension conjecture. Namely, functions on the circle which
approximate two Dirac deltas at antipodally symmetric points are close
competitors to extremize the Tomas-Stein functional; they ``lose'' to
the constant function by only a small amount. Such Dirac deltas, on
the Fourier transform side depicted in the above hexagons, correspond
to wide bumps such as the lowest eigenfunction. One can well imagine
that all the radial eigenfunctions to small eigenvalues aspire to
resolve structure near the Dirac deltas.

\section*{Acknowledgements}
The main calculations described in this paper were performed using the
supercomputing facilities of Fraunhofer SCAI\@. The authors
acknowledge support by the Deutsche Forschungsgemeinschaft through the
Hausdorff Center for Mathematics (DFG Projektnummer 390685813) and the
Collaborative Research Center 1060 (DFG Projektnummer 211504053). The
authors also gratefully acknowledge the comments and suggestions of
the anonymous reviewers, and in particular their detection of several
small errors.

\printbibliography{}

\end{document}